\documentclass{article}

\usepackage[
	a4paper,
	margin=1in
]{geometry}
\usepackage{setspace}
\usepackage{sectsty} 
\usepackage{titlesec}
\usepackage{appendix}

\usepackage{amsmath, amssymb, amsthm}
\usepackage{mathtools}

\usepackage[
	setpagesize=false,
	colorlinks=true,
	linkcolor=BrickRed,
        citecolor=OliveGreen,
        urlcolor=black,
	pdfencoding=auto,
	psdextra,
]{hyperref}
\usepackage{cleveref}

\usepackage{etoolbox} 
\usepackage[shortlabels]{enumitem}
\usepackage{xparse} 

\usepackage{graphicx}
\usepackage[dvipsnames]{xcolor}
\usepackage{tikz}

\usepackage{todonotes}
\usepackage[
    deletedmarkup=sout,
    commentmarkup=uwave,
    authormarkuptext=name
]{changes}

\usepackage{comment}

\setlist[enumerate,1]{label=(\arabic*), ref=(\arabic*)}
\setlist[enumerate,3]{label=(\roman*), ref=(\roman*)}

\allsectionsfont{\boldmath} 

\titlelabel{\thetitle.\quad}

\theoremstyle{plain}
\newtheorem{theorem}{Theorem}[section]
\newtheorem{lemma}[theorem]{Lemma}

\newtheorem{question}[theorem]{Question}

\newtheorem{claim}[theorem]{Claim}
\newtheorem*{claim*}{Claim}

\makeatletter
\newenvironment{claimproof}[1][Proof]{\par
	\pushQED{\qed}%
	
	\normalfont \topsep6\p@\@plus6\p@\relax
	\trivlist
	\item[\hskip\labelsep
	\textit{#1}\@addpunct{.}~]\ignorespaces
}{%
	\popQED\endtrivlist\@endpefalse
}
\makeatother

\newlist{Cases}{enumerate}{3}
\setlist[Cases]{parsep=0pt plus 1pt}
\setlist[Cases,1]{wide=0pt, listparindent=\parindent,
    label = \textbf{Case~\arabic*:}, ref = \arabic*}
\setlist[Cases,2]{wide=\parindent, listparindent=\parindent,
    label = \textbf{Case~\arabic{Casesi}-\arabic{Casesii}:}}

\crefname{Casesi}{case}{cases}

\newcounter{case}
\AtBeginEnvironment{proof}{\setcounter{case}{0}}

\crefname{case}{case}{cases}

\theoremstyle{definition}
\newtheorem{definition}[theorem]{Definition}

\newcommand{\calC}{\mathcal{C}}

\newcommand{\calT}{\mathcal{T}}

\newcommand{\ve}{\varepsilon}
\newcommand{\abs}[1]{\left\lvert#1\right\rvert}

\makeatletter
\NewDocumentCommand{\xsideset}{mmme{_^}}{%
  \mathop{%
    \settowidth{\dimen0}{$\m@th\displaystyle#3$}%
    \dimen0=.5\dimen0
    \settowidth{\dimen2}{$%
      \m@th\displaystyle#3%
      \IfValueT{#4}{_{#4}}%
      \IfValueT{#5}{^{#5}}%
    $}%
    \dimen2=.5\dimen2
    \advance\dimen2 -\dimen0
    \sbox6{\scriptspace\z@$\displaystyle{\vphantom{#3}}#1$}
    \sbox8{\scriptspace\z@$\displaystyle{\vphantom{#3}}#2$}
    \ifdim\wd6>\dimen2 \kern\dimexpr\wd6-\dimen2\relax\fi
    {%
     \mathop{\llap{\copy6}{\displaystyle#3}\rlap{\copy8}}\limits
     \IfValueT{#4}{_{#4}}%
     \IfValueT{#5}{^{#5}}%
    }%
    \ifdim\wd8>\dimen2 \kern\dimexpr\wd8-\dimen2\relax\fi
  }%
}
\makeatother

\newcommand{\defeq}{\coloneqq}

\let\originalleft\left
\let\originalright\right
\renewcommand{\left}{\mathopen{}\mathclose\bgroup\originalleft}
\renewcommand{\right}{\aftergroup\egroup\originalright}

\makeatletter
\newcases{lrcases*}
  {\quad}
  {$\m@th{##}$\hfil}
  {{##}\hfil}
  {\lbrace}
  {\rbrace}
\makeatother

\usetikzlibrary{matrix,arrows,decorations.pathmorphing}
\usetikzlibrary{positioning}
\usetikzlibrary{graphs} 
\usetikzlibrary{graphs.standard}

\title{A quantitative improvement on the hypergraph Balog--Szemer\'{e}di--Gowers theorem}

\author{
Hyunwoo Lee%
        \thanks{Department of Mathematical Sciences, KAIST, South Korea and Extremal Combinatorics and Probability Group (ECOPRO), Institute for Basic Science (IBS).
        E-mail: {\ttfamily hyunwoo.lee@kaist.ac.kr.} Supported by the National Research Foundation of Korea (NRF) grant funded by the Korea government(MSIT) No. RS-2023-00210430, and the Institute for Basic Science (IBS-R029-C4).}
}

\usepackage[
    backend=biber,
    sorting=nyt,
    maxbibnames=99
]{biblatex}
\begin{document}
\maketitle

\begin{abstract}
    In this note, we obtain a quantitative improvement on the hypergraph variant of the Balog--Szemer\'{e}di--Gowers theorem due to Sudakov, Szemer\'{e}di, and Vu [Duke Math. J.129.1 (2005): 129--155]. Additionally, we prove the hypergraph variant of the ``almost all'' version of Balog--Szemer\'{e}di--Gowers theorem.
\end{abstract}


\section{Introduction}\label{sec:intro}

For an abelian group $\mathbf{G}$ and its subset $A$, we define the \emph{doubling constant} $\sigma(A) \defeq \frac{\abs{A + A}}{\abs{A}}$. Sets with small doubling constants exhibit rich structural properties, making them a central topic in additive combinatorics. For instance, if a set $A$ has doubling constant $C$, then the number of quadruple $(x, y, x', y') \in A^4$ satisfying $x + y = x' + y'$ is at least $\frac{\abs{A}^3}{C}$. The number of such quadruples is called \emph{additive energy}. In various topics related to additive combinatorics, a set with high additive energy plays a key role. As additive energy is a fundamental concept, it is natural to ask whether large additive energy implies a small doubling constant.
However, this does not hold in general. For example, consider a set of $n$ integers where half of them form an arithmetic progression and the rest are random integers. Then this set has large additive energy as well as a large doubling constant. On the other hand, while this set has a large doubling constant, it contains a subset with a small doubling constant, which is an arithmetic progression. Indeed, this is not a coincidence as the following celebrated Balog--Szemer\'{e}di-Gowers theorem~\cite{balog-szemeredi,Gowers-BSG} asserts,

\begin{theorem}[Balog-Szemer\'{e}di-Gowers~\cite{balog-szemeredi,Gowers-BSG}]\label{thm:BSG}
    Let $\mathbf{G}$ be an abelian group and $A \subseteq \mathbf{G}$. If $A$ has an additive energy at least $\frac{|A|^3}{C}$ for some constant $C > 0$, then there is a subset $A'\subseteq A$ with $\abs{A'} \geq \Omega_C(1) \abs{A}$ and $\sigma(A') \leq O_C(1)$. 
\end{theorem}

Balog and Szemer\'{e}di~\cite{balog-szemeredi} first proved \Cref{thm:BSG} by using the regularity lemma, and Gowers~\cite{Gowers-BSG} significantly refined their result by improving the implicit constants to polynomials of $C$. A further refinement was obtained by Sudakov, Szemer\'{e}di, and Vu~\cite{Sudakov-Szemeredi-Vu}, who introduced the graph version of Balog-Szemer\'{e}di-Gowers theorem. For subsets $A$ and $B$ of an abelian group $\mathbf{G}$ and a bipartite graph $H \subseteq A\times B$, we denote $\bigoplus_H (A, B) \defeq \{a+b : (a, b) \in E(H)\}$. The modern formulation of the graph version of \Cref{thm:BSG} is as follows.

\begin{theorem}[Sudakov--Szemer\'{e}di--Vu~{\cite[Theorem 4.1]{Sudakov-Szemeredi-Vu}}, Tao-Vu~{\cite[Theorem 2.29]{Tao-Vu-book}}]\label{thm:graph-BSG}
    Let $A$ and $B$ be subsets of an abelian group $\mathbf{G}$ and $H \subseteq A\times B$ be a bipartite graph with $\abs{E(H)} \geq \frac{\abs{A} \abs{B}}{K}$ and $\abs{\bigoplus_H (A, B)} \leq C|A|^{1/2}|B|^{1/2}$. Then there exist $A' \subseteq A$ and $B'\subseteq B$ such that 
    \begin{enumerate}
        \item[$\bullet$] $\abs{A'} \geq \frac{\abs{A}}{4\sqrt{2}K}$ and $\abs{B'} \geq \frac{\abs{B}}{4K}$,
        \item[$\bullet$] $\abs{A' + B'} \leq 2^{12}K^4 C^3 \abs{A}^{1/2}\abs{B}^{1/2}$.
    \end{enumerate}
\end{theorem}
We remark that \Cref{thm:graph-BSG} implies \Cref{thm:BSG} via standard arguments using tools such as Ruzsa triangle inequality~\cite{Tao-Vu-book} and Pl\"{u}nnecke-Ruzsa inequality~\cite{plunnecke-ruzsa}.

Sudakov, Szemer\'{e}di and Vu~\cite{Sudakov-Szemeredi-Vu} obtained a further generalization of \Cref{thm:graph-BSG}, which considers $r$-partite $r$-uniform sums rather than just bipartite sums.
For subsets $A_1, \dots, A_r$ of an abelian group $\mathbf{G}$ and an $r$-partite $r$-uniform hypergraph $H \subseteq A_1 \times \cdots \times A_r$, we denote $\bigoplus_H (A_1, \dots, A_r)\defeq \{\sum_{i\in [r]} a_i : (a_1, \dots, a_r) \in E(H)\}$.
The following is the hypergraph version of Balog--Szemer\'{e}di--Gowers theorem.

\begin{theorem}[Sudakov-Szemer\'{e}di-Vu~{\cite[Theorem 4.3]{Sudakov-Szemeredi-Vu}}]\label{thm:hypergraph-BSG}
    Let $r\geq 2$ be an integer. Let $A_1, \dots, A_r$ be sets of $n$ positive integers and $H \subseteq A_1 \times \cdots \times A_r$ be an $r$-partite $r$-uniform hypergraph with $\abs{E(H)} \geq \frac{n^r}{K}$ and $\abs{\bigoplus_H (A_1, \dots, A_r)} \leq C n$. Then for each $i\in [r]$, there exist $A_i' \subseteq A_i$ such that 
    \begin{enumerate}
        \item[$\bullet$] $\abs{A_i'} \geq \frac{n}{f_r(C, K)}$ for each $i\in [r]$,
        \item[$\bullet$] $\abs{A_1' + \cdots + A_r'} \leq g_r(C, K) n$,
    \end{enumerate}
    where $f_r$ and $g_r$ are two-variable polynomials whose degrees and coefficients depend only on $r$.
\end{theorem}

Our main theorem is the quantitative improvement of \Cref{thm:hypergraph-BSG}, as follows.

\begin{theorem}\label{thm:main}
    Let $r \geq 2$ be an integer. Let $A_1, \dots, A_r$ be subsets of an abelian group $\mathbf{G}$ and $H \subseteq A_1 \times \cdots \times A_r$ be an $r$-partite $r$-uniform hypergraph with $\abs{E(H)} \geq \frac{1}{K} \prod_{i\in [r]} \abs{A_i}$ and $\abs{\bigoplus_H (A_1, \dots, A_r)} \leq C\left(\prod_{i \in [r]} \abs{A_i} \right)^{1/r}$. Then for each $i\in [r]$, there exist $A_i' \subseteq A_i$ such that 
    \begin{enumerate}
        \item[$\bullet$] $\abs{A'_i} \geq \frac{\abs{A_i}}{2^{i+2} K}$ for each $i\in [r]$,
        \item[$\bullet$] $\abs{A_1' + \cdots + A_r'} \leq 8^{r^3}(r-1)^{(r-1)}K^{(r^2 + 5r - 4)/2}C^{2r-1} \left(\prod_{i\in [r]} \abs{A_i} \right)^{1/r}$.
    \end{enumerate}
\end{theorem}

Note that \Cref{thm:main} considerably improves the quantitative bound of \Cref{thm:hypergraph-BSG} in several ways. In particular, we replace $f(C, K)$ with a linear function of $K$, and improve the degrees of $g_r$. For constant $K$, the proof in~\cite{Sudakov-Szemeredi-Vu} yield that the polynomial $g_r(C)$ has degree $\Omega(r^2)$.
On the other hand, \Cref{thm:main} reduces the degree to be $O(r)$. Lastly, our main theorem does not need the condition that $|A_1| = \cdots = |A_r|$. This allows us to apply \Cref{thm:main} in asymmetric cases.

We also prove an ``almost all" version of the hypergraph Balog-Szemer\'{e}di-Gowers theorem by employing the proof method of \Cref{thm:main} with Shao's argument~\cite{Shao}.

\begin{theorem}\label{thm:almost-hypergraph-BSG}
    Let $0  < \ve < \frac{1}{2} < 1 < C$ be positive real numbers and $r \geq 2$ be an integer. Then there exists $\delta = \delta(r, C, \ve) > 0$ such that the following holds. 
    Let $A_1, \dots, A_r$ be subsets of an abelian group $\mathbf{G}$ with $\lvert A_1 \rvert = \cdots = \lvert A_r \rvert = n$. Let $H \subseteq A_1 \times \cdots \times A_r$ be an $r$-uniform hypergraph with $\lvert E(H) \rvert \geq (1 - \delta)n^r$ and $\lvert \bigoplus_H (A_1, \dots, A_r) \rvert \leq C n$.
    Then for each $i\in [r]$, there exists $A'_i \in A_i$ such that 
    \begin{enumerate}
        \item[$\bullet$] $\lvert A_i'\rvert \geq (1 - \ve)n$ for each $i\in [r]$,
        \item[$\bullet$] $\lvert A_1' + \cdots + A_r'\rvert \leq (C + \ve) n$.
    \end{enumerate}
\end{theorem}

Note that Shao~\cite{Shao} discovered a beautiful connection between \Cref{thm:graph-BSG} and the arithmetic removal lemma and proved the $r = 2$ case of \Cref{thm:almost-hypergraph-BSG}.

In~\cite{Sudakov-Szemeredi-Vu}, \Cref{thm:hypergraph-BSG} is proved by induction on the uniformity $r$ of the hypergraph $H$. On the other hand, we directly generalize the proof of \Cref{thm:graph-BSG} without using induction.

\paragraph{Remark $1$.} There is an another hypergraph variants of \Cref{thm:graph-BSG} due to Borenstein and Croot~\cite{Borenstein-Croot}. Their result provides more information on the growth of iterative sumsets but only works for symmetric cases, $A_1 = \cdots = A_r$. 

\paragraph{Remark $2$.}
A quantitative version of Borenstein and Croot~\cite{Borenstein-Croot} is proved by Mudgal~\cite{Mudgal-private}, which can be deduced from~\cite[Theorem 6.2]{Mudgal}.
\begin{theorem}[Mudgal~\cite{Mudgal,Mudgal-private}]\label{thm:mudgal}
    Let $\mathbf{G}$, $A_1, \dots, A_r$, $H$, and constants $K, C$ are the same as in the statement of \Cref{thm:main} with $A_1 = \cdots = A_r = A$.  Then there exist $A' \subseteq A$ such that 
    \begin{enumerate}
        \item[$\bullet$] $\abs{A'} \geq \frac{D'\abs{A}}{(CK)^{D'/\log r}}$,
        \item[$\bullet$] $\abs{r\dot A'} = \abs{A' + \cdots + A'} \leq D_r (CK)^{D(1 + r/\log r)}\abs{A'}$,
    \end{enumerate}
    where $D, D'$ are absolute constants and $D_r$ is a constant only depending on $r$.
\end{theorem}
Note that \Cref{thm:mudgal} gives a better quantitative bound for larger $r$ with a smaller $C$, yet our main result \Cref{thm:main} does not require the symmetric condition and also completely removes the dependency of $C$ in the size of $A'$.

\subsection{Proof overview}\label{subsec:overview}

First, we briefly discuss how one can prove \Cref{thm:graph-BSG}. 
The key lemma to prove \Cref{thm:graph-BSG} is that any bipartite graph $G \subseteq A\times B$ on the vertex set $A\cup B$ with $\abs{E(G)} \geq \frac{\abs{A}\abs{B}}{K}$ contains two not--so--small vertex subsets $A' \subseteq A$ and $B' \subseteq B$ such that for any $(a, b)\in A'\times B'$, there are many copies of path of length three from $a$ to $b$ in $G$. 
Once we have these sets $A'$ and $B'$, any $(a, b)\in A' \times B'$ can be written $a + b = (a + b') - (a' + b') + (a' + b)$ for many $a'\in A$ and $b'\in B$, where $ab'a'b$ forms a path of length three in $G$. Since $(a, b'), (a', b'), (a', b) \in E(H)$, we can write $a + b = x_1 - x_2 + x_3$, where $x_1 = a + b', x_2 = a' + b', x_3 = a' + b$ are all in $\bigoplus_G (A, B)$. 
Note that once $a, b$ are fixed, the number of such $(x_1, x_2, x_3)$ is equal to the number of $(a', b')$.
Since $\bigoplus_G (A, B)$ is small and each $a+b$ admits many different expressions in the forms of $x_1 - x_2 + x_3$, the set $A' + B'$ must be small. This summarizes the proof of \Cref{thm:graph-BSG}.

Our purpose is to generalize the above strategy to hypergraphs. In order to achieve this, we need to find an appropriate structure corresponding to the path of length three. We call such a structure \emph{$r$-octopus} which is defined at \Cref{def:octopus}.
With this definition on our hand, we prove that for a given $r$-partite $r$-uniform hypergraph, we can find large vertex subsets such that every $r$-set crossing them supports many $r$-octopuses.
Once we have many $r$-octopuses, we can show that each element in the subset of the vertex subsets can be written in many different ways by using the members in the partial sumsets restricted to the hypergraph $H$. Finally, using a counting argument (unlike the graph case, we need more careful analysis in this step), we can deduce our main theorems. 


\section{Large sets supporting many octopuses}\label{sec:octopuses}

Since all graphs and hypergraphs discussed in this paper are bipartite graphs or $r$-partite $r$-uniform hypergraphs, we assume an $r$-partition $V_1, \dots, V_r$ is given and edges are ordered tuples in $V_1 \times \cdots \times V_r$ rather than subsets of vertices for convenience in notation.

As we discussed in \Cref{subsec:overview}, we define $r$-octopus and show that this acts like the path of length three in some sense. The definitions are as follows.

\begin{definition}\label{def:leg}
    Let $r\geq 2$ be an integer and assume a vertex partition $V_1, \dots, V_r$ of $V$ is given. For $i\in [r]$ and distinct vertices $v, w\in V_i$, we say that 
    an $r$-uniform hypergraph $L$ with
    $$V(L) = \{v, w\} \cup \{u_j: j\in [r]\setminus \{i\}\}$$ and $$E(L) = (u_1, \dots, u_{i-1}, v, u_{i+1}, \dots, u_{r}) \cup (u_1, \dots, u_{i-1}, w, u_{i+1}, \dots, u_{r}) \subseteq V_1 \times \cdots \times V_r$$
    is an \emph{$i$-th $r$-leg on $(v, w)$}. 
\end{definition}

\begin{definition}\label{def:octopus}
    Let $r \geq 2$ be an integer and assume a vertex partition $V_1, \dots, V_r$ of $V$ is given. Let $v_1, \dots, v_r, w_1 \dots, w_{r-1}$ be distinct vertices with $v_i\in V_i$ and $w_j \in V_j$, where $i\in [r]$ and $j\in [r-1]$. For each $i\in [r-1]$, let $L_i$ be a vertex-disjoint $i$-th $r$-leg on $(v_i, w_i)$. 
    We say that an $r$-partite $r$-uniform hypergraph $T$ with
    $$V(T) = \bigcup_{i\in [r-1]} V(L_i) \cup \{v_r\}$$ and $$E(T) = \bigcup_{i\in [r-1]} E(L_i) \cup \{(w_1, \dots, w_{r-1}, v_r)\}$$
    is an \emph{$r$-octopus supported on $(v_1, \dots, v_r)$}.
\end{definition}

Note that a $2$-leg is a path of length two and a $2$-octopus is a path of length three in the graph case.

We now claim that, as in the graph case, there exist large vertex subsets that every tuple of vertices supports many $r$-octopuses in an $r$-partite $r$-uniform hypergraph. We first show this in general in the following section, and then we show that one can even ensure that such a vertex subset is almost spanning if the given hypergraph is very dense.

\subsection{General case}\label{subsec:general}
The purpose of this section is to prove the following lemma.

\begin{lemma}\label{lem:octopuses}
    Let $A_1, \dots, A_r$ be disjoint vertex sets and $H\subseteq A_1 \times \cdots \times A_r$ be an $r$-partite $r$-uniform hypergraph with $\abs{E(H)} \geq \frac{1}{K}\prod_{i\in [r]} \abs{A_i}$. Then for each $i\in [r]$, there exists $A'_i\subseteq A_i$ satisfying the following.
    \begin{enumerate}
        \item[$(L1)$] $\abs{A'_i} \geq \frac{\abs{A_i}}{2^{i+2} K}$,
        \item[$(L2)$] for each $(a_1, \dots, a_r) \in A'_1 \times \cdots \times A'_r$, the number of $r$-octopuses supported on $(a_1, \dots, a_r)$ in $H$ is at least $$\frac{1}{8^{r^3}(r-1)^{(r-1)} K^{(r^2 + 5r - 4)/2}} \times \left(\prod_{i\in [r]} \abs{A_i}\right)^{r-1}.$$
    \end{enumerate}
\end{lemma}

In order to prove \Cref{lem:octopuses}, we collect several lemmas. The following lemma can be proved by using \emph{dependent random choice} which is a powerful and well-known technique in extremal combinatorics. 

\begin{lemma}[{\cite[Lemma 5.1]{DRC}}]\label{lem:drc}
    Let $A, B$ be two disjoint vertex sets and $H \subseteq A\times B$ be a bipartite graph with $\lvert E(H) \rvert \geq \frac{\lvert A \rvert \lvert B \rvert}{K}$. Let $0 < \ve < 1$ be an arbitrary real number.
    Then there exists $U\subseteq A$ such that
    \begin{enumerate}
        \item[$\bullet$] $\lvert U \rvert \geq \frac{\lvert A \rvert}{2K}$,
        \item[$\bullet$] at least a $(1 - \ve)$-fraction of the ordered pairs of vertices in $U$ have at least $\frac{\ve \lvert B \rvert}{2K^2}$ common neighbors in $B$.
    \end{enumerate}
\end{lemma}

Using \Cref{lem:drc}, we can prove the following lemma which is crucial for the proof of \Cref{lem:octopuses}.

\begin{lemma}\label{lem:iterate}
    Let $r \geq 2$ be an integer and $V_1, \dots, V_r$ be disjoint vertex sets and $H\subseteq V_1 \times \cdots \times V_t$ be an $r$-partite $r$-uniform hypergraph with $\lvert E(H) \rvert \geq \frac{1}{K}\prod_{i\in [r]} \abs{V_i}$. Let $0 < \ve < 1$ be an arbitrary real number. Then there exists $U\subseteq V_1$ such that
    \begin{enumerate}
        \item[$(a)$] $\lvert U \rvert \geq \frac{\abs{V_1}}{4K}$,
        \item[$(b)$] at least a $(1 - \ve)$-fraction of the ordered pairs of vertices $(v, w)$ in $U$ have at least $$\frac{\ve}{2K^2} \prod_{2 \leq i \leq r} \abs{V_i}$$ $1$-th $r$-legs on $(v, w)$ in $H$,
        \item[$(c)$] for each vertex $u\in U$, the vertex $u$ has degree at least
        $$
        \frac{1}{2K}\prod_{2 \leq i \leq r} \abs{V_i}.
        $$
    \end{enumerate}
\end{lemma}

\begin{proof}[Proof of \Cref{lem:iterate}]
    We construct an auxiliarly bipartite graph $G$ on the bipartition $V_1$ and the set $Z\defeq V_2 \times \cdots \times V_r$ by joining two vertices $(a, b)\in V_1 \times Z$ if $(a, b) \in E(H)$. Then we have $\abs{E(G)} \geq \frac{\abs{V_1} \abs{Z}}{K}$.

    Let $V'_1$ be the set of vertices in $V_1$ of degree at least $\frac{\abs{Z}}{2K}$ and $G'$ be the induced graph of $G$ on the vertex set $V'_1 \cup Z$. Then the following inequality holds. 
    \begin{equation}\label{eq:1}
        \abs{E(G')} \geq \abs{E(G)} - \frac{\abs{V_1}\abs{Z}}{2K} \geq \frac{\abs{V_1}\abs{Z}}{2K}.
    \end{equation}
    Hence, we have the following.
    \begin{equation}\label{eq:2}
        \frac{1}{K'} \defeq \frac{\abs{E(G')}}{\abs{V'_1}\abs{Z}} \geq \frac{\abs{V_1}\abs{Z}}{2K\abs{V'_1}\abs{Z}} = \frac{\abs{V_1}}{2K \abs{V'_1}}.
    \end{equation}
    Moreover, since $G'$ is obtained from $G$ by deleting small degree vertices in $V_1$, the edge-density of $G'$ with respect to the partition is at least the edge-density of $G$ with respect to the partition, we have 
    \begin{equation}\label{eq:3}
        \frac{1}{K'} \geq \frac{1}{K}.
    \end{equation}

    By applying \Cref{lem:drc} on $G'$ and the vertex sets $V'_1, Z$ and constants $K'$ and $\ve$, we obtain a vertex subset $U \subseteq V'_1$ such that 
    \begin{enumerate}
        \item[$(i)$] $\abs{U} \geq \frac{\abs{V'_1}}{2K'}$,
        \item[$(ii)$] at least a $(1 - \ve)$-fraction of the ordered pairs of vertices of $(v, w)$ in $U$ have at least $\frac{\ve \abs{Z}}{2 K'^2}$ common neighbors in $Z$.
    \end{enumerate}

    We now claim that $U$ is the desired subset of $V_1$. By \eqref{eq:2} and $(i)$, we have $\abs{U} \geq \frac{\abs{V_1}}{4K}$, which implies $(a)$. By \eqref{eq:3} and $(ii)$, we have $\frac{\ve \abs{Z}}{2K'^2} \geq \frac{\ve \abs{Z}}{2K^2}$. Since $Z = \prod_{2 \leq i \leq r} V_i$ and the path of length two in $G$ corresponds to the $1$-th $r$-leg in $H$, this implies $U$ also satisfies $(b)$. Finally, the set $U$ is a subset of $V'_1$, thus for each $u\in U$, the degree of $u$ is at least $\frac{\abs{Z}}{2K}$.
    We note that there is an obvious one-to-one correspondence between the edges of $G''$ and the edges of an induced hypergraph of $H$ on the vertex set $U \cup \bigcup_{2 \leq i \leq r} V_i$. Since $\abs{Z} = \prod_{2 \leq i \leq r} \abs{V_i}$, the set $U$ satisfies $(c)$. This completes the proof.
\end{proof}

Note that by symmetry, \Cref{lem:iterate} also holds for $i$-th $r$-legs for any $i\in [r]$.
We are now ready to prove \Cref{lem:octopuses}.

\begin{proof}[Proof of \Cref{lem:octopuses}]
    For each $i\in [r-1]$ and a real number $0 < \ve < 1$, we say an ordered pair $(v, w) \in A_i\times A_i$ is \emph{$\ve$-good} if the number of $i$-th $r$-legs on $(v, w)$ in $H$ is at least $$\frac{\ve}{2^{r^2} K^{i+1}}\prod_{j\in [r]\setminus \{i\}}\abs{A_i}.$$
    For a subset $U \subseteq A_i$ and a real number $0 \leq \ve' < 1$, we say a vertex $v\in A_i$ is \emph{$(\ve, \ve'; U)$-good} if the number of vertices $w\in U$ such that $(v, w)$ is a $\ve$-good pair is at least $(1 - \ve')\abs{U}$.
    
    We start the proof with the following claim.
    
    \begin{claim}\label{clm:1}
        Let $0 < \ve < \frac{1}{4}$ be an arbitrary real number.
        For each $0 \leq i \leq r-1$, there exists $A'_j \subseteq A_j$ for each $1 \leq j \leq i$ and $H_i \subseteq H$ that satisfy the following.
        \begin{enumerate}
            \item[$(P1)$] $\abs{A'_j} \geq \frac{\abs{A_j}}{2^{j+2} K}$,
            \item[$(P2)$] all vertices in $A'_j$ are $(\ve, 4\ve; A'_j)$-good, 
            \item[$(P3)$] $V(H_i) = \bigcup_{1 \leq k \leq i} A'_k \cup \bigcup_{i+1 \leq \ell \leq r} A_{\ell}$,
            \item[$(P4)$] $\abs{E(H_i)} \geq \frac{1}{2^i K}\prod_{1\leq k \leq i}\abs{A'_k} \prod_{i+1 \leq \ell \leq r}\abs{A_{\ell}}.$
        \end{enumerate}
    \end{claim}       

    \begin{claimproof}[Proof of \Cref{clm:1}]
        We use an induction on $i$. The base case, $i = 0$, we set $H_0$ as $H$. Then $(P1)\text{--}(P4)$ holds by the initial condition of the statement of \Cref{lem:octopuses}. 
        
        For some $0 \leq i \leq r-2$, assume there exist $A'_1, \dots, A'_{i}$ and $H_{i}$ that satisfies $(P1)\text{--}(P4)$. We apply \Cref{lem:iterate} $H_i$, where $H_i, A_{i+1}, 2^i K$ plays a role of $H, V_1, K$, respectively. Then there exists a subset of $A_{i+1}$ that satisfies the conclusions of \Cref{lem:iterate}. We set such set as $\widetilde{A}_{i+1}$. 
        That is, we have $\widetilde{A}_{i+1} \subseteq A_{i+1}$ such that 
        $$
        \abs{\widetilde{A}_{i+1}} \geq \frac{\abs{A_{i+1}}}{2^{i+2}K}
        $$ by $(a)$ of \Cref{lem:iterate} and at least a $(1 - \ve)$-fraction of the ordered pairs of vertices $(v, w)$ in $\widetilde{A}_{i+1}$ have at least 
        $$
        \frac{\ve}{2^{2i+1}K^2}\prod_{1 \leq k \leq i} \abs{A'_{k}}\prod_{i+2 \leq \ell \leq r} \abs{A_{\ell}}
        $$
        $(i+1)$-th $r$-legs on $(v, w)$ in $H_i$ by $(b)$ of \Cref{lem:iterate}.
        By the induction hypothesis, for each $k\in [i]$, we have $\abs{A'_k} \geq \frac{\abs{A_k}}{2^{k+2}K}$. As a result, the following holds.
        $$
        \frac{\ve}{2^{2i+1}K^2}\prod_{1 \leq k \leq i} \abs{A'_{k}}\prod_{i+2 \leq \ell \leq r} \abs{A_{\ell}} \geq \frac{\ve}{2^{(i^2 + 9i + 2)/ 2} K^{i+2}} \prod_{\ell \in [r]\setminus \{i+1\}} \abs{A_{\ell}} \geq \frac{\ve}{2^{r^2} K^{i+2}} \prod_{\ell \in [r]\setminus \{i+1\}} \abs{A_{\ell}}.
        $$
        This implies at least a $(1 - \ve)$-fraction of the ordered pairs of vertices in $\widetilde{A}_{i+1}\times \widetilde{A}_{i+1}$ are $\ve$-good.
        Let $A'_{i+1} \subseteq \widetilde{A}_{i+1}$ be the set of vertices $v$ in $\widetilde{A}_{i+1}$ such that at least $(1 - 2\ve)\abs{\widetilde{A}_{i+1}}$ of vertex $w \in \widetilde{A}_{i+1}$ satisfies that $(v, w)$ is a $\ve$-good. Then by Markov's inequality, we have 
        \begin{equation}\label{eq:markov}
            \abs{A'_{i+1}} \geq \frac{\abs{\widetilde{A}_{i+1}}}{2} \geq \frac{\abs{A_{i+1}}}{2^{i+3}K}.
        \end{equation}
        Since for each $u\in A'_{i+1}$, there is at most $2\ve \abs{\widetilde{A}_{i+1}}$ vertices $w$ in $\widetilde{A}_{i+1}$ that do not form a $(\ve, i+1)$-good pair with $v$, and the size $A'_{i+1}$ is at least the half of $\abs{\widetilde{A}_{i+1}}$, there are at most $4\ve \abs{A'_{i+1}}$ vertices in $A'_{i+1}$ that do not form $\ve$-good pairs with $v$. This means all vertices in $A'_{i+1}$ is $(\ve, 4\ve; A'_{i+1})$-good. 

        We now claim that $A'_{i+1}$ is the desired set. By \eqref{eq:markov} and the above paragraph, the set $A'_{i+1}$ satisfies $(P1)$ and $(P2)$. 
        
        To check $A'_{i+1}$ also satisfies $(P3)$ and $(P4)$ is simple. We now define $H_{i+1}$ as the induced sub-hypergraph of $H_i$ on the vertex set 
        $$
        \bigcup_{1 \leq k \leq i+1} A'_k \cup \bigcup_{i+2 \leq \ell \leq r} A_{\ell}.
        $$ Then obviously $H_{i+1}$ satisfies $(P3)$. Finally, by $(c)$ of \Cref{lem:iterate},
        every vertex $u\in A'_{i+1} \subseteq \widetilde{A}_{i+1}$ has degree at least 
        $$
        \frac{1}{2^{i+1}K}\prod_{1 \leq k \leq i} \abs{A'_i}\prod_{i+2 \leq \ell \leq r}\abs{A_{\ell}}
        $$ in $H_i$.
        From this, the last property $(P4)$ is immediately deduced. By induction, this completes the proof.
    \end{claimproof}

    Let $0 < \ve < \frac{1}{4}$ be a real number that will be determined later.
    Let $\widetilde{H}$ be the hypergraph $H_{r-1}$ from \Cref{clm:1} and for each $i\in [r-1]$, let $A'_i \subseteq A_i$ be the sets guaranteed from \Cref{clm:1}. Then by $(P4)$, we have
    \begin{equation}\label{eq:4}
        \abs{E(\widetilde{H})} \geq \frac{1}{2^{r-1} K} \abs{A_r} \prod_{i\in [r-1]}\abs{A'_i}.
    \end{equation}
    Let $A'_r \subseteq A_r$ be the set of vertices in $A_r$ of degree at least $\frac{1}{2^r K} \prod_{i\in [r-1]}\abs{A'_i}.$ Then we have
    \begin{equation}\label{eq:5}
        \abs{E(\widetilde{H})} \leq \abs{A'_r} \prod_{i\in [r-1]} \abs{A'_i} + \frac{1}{2^r K} \abs{A_r} \prod_{i\in [r-1]}\abs{A'_i}.
    \end{equation}
    By combining \eqref{eq:4} and \eqref{eq:5}, the following holds.
    \begin{equation}\label{eq:6}
        \abs{A'_r} \geq \frac{\abs{A_r}}{2^r K} \geq \frac{\abs{A_r}}{2^{r + 2}K}.
    \end{equation}

    We now set 
    $$
    \ve \defeq \frac{1}{(r-1) 2^{r+3} K}.
    $$
    The rest of the proof is to show that $A'_1, \dots, A'_r$ are the desired sets. Note that by $(P1)$ of \Cref{clm:1} and \eqref{eq:6}, the sets $A'_1, \dots, A'_r$ satisfy $(L1)$ of \Cref{lem:octopuses}. Let $H'\subseteq \widetilde{H}$ be the induced hypergraph on the vertex set $\bigcup_{i\in [r]} A'_i$. 

    \begin{claim}\label{clm:2}
        For each $(a_1, \dots, a_r)\in A'_1 \times \cdots \times A'_r$, the number of $r$-octopuses in $H$ that are supported on $(a_1, \dots, a_r)$ is at least $$\frac{1}{8^{r^3}(r-1)^{(r-1)} K^{(r^2 + 5r - 4)/2}} \times \left(\prod_{i\in [r]} \abs{A_i}\right)^{r-1}.$$
    \end{claim}

    \begin{claimproof}[Proof of \Cref{clm:2}]
        By our choice of the set $A'_r$, the vertex $a_r$ has degree at least $\frac{1}{2^r K}\prod_{i\in [r-1]} \abs{A'_i}$ in $H'$. 
        Let $H^- \subseteq A'_1 \times \cdots \times A'_{r-1}$ be an $(r-1)$-partite $(r-1)$ uniform hypergraph such that 
        $$
        V(H^-) = \bigcup_{i\in [r-1]} A'_i \text{ and } E(H^-) = \{e\in A'_1 \times \cdots \times A'_{r-1}: (e, a_r) \in E(H')\}.
        $$
        As $\abs{E(H^-)}$ is same with the degree of $a_r$, we have 
        \begin{equation}\label{eq:8}
            \abs{E(H^-)} \geq \frac{1}{2^rK}\prod_{i\in [r-1]}\abs{A'_i}.
        \end{equation}

        For each $i\in [r-1]$, let $U_i$ be the set of $u\in \widetilde{A}_i$ such that $(a_i, u)$ is $\ve$-good. By $(P2)$ of \Cref{clm:1}, for each $i\in [r-1]$, we have 
        \begin{equation}\label{eq:9}
            \abs{U_i} \geq (1 - 4\ve) \abs{A'_i}.
        \end{equation}

        By our choice of $\ve$ and \eqref{eq:8}, \eqref{eq:9}, the following holds.

        \begin{equation}\label{eq:10}
            \abs{E(H^-) \cap (U_1 \times \cdots \times U_{r-1})} \geq \frac{1}{2^rK}\prod_{i\in [r-1]}\abs{A'_i} - 4(r-1)\ve\prod_{i\in [r-1]}\abs{A'_i} \geq \frac{1}{2^{r+1}K}\prod_{i\in [r-1]}\abs{A'_i}.
        \end{equation}
        We note that for each $i\in [r-1]$ and $u\in U_i$, since $(a_i, u)$ is an $\ve$-good pair, the number of $i$-th $r$-legs in $H$ on $(a_i, u)$ is at least $$\frac{\ve}{2^{r^2}K^{i+1}}\prod_{j \in [r]\setminus\{i\}}\abs{A_j}.$$ Together with \eqref{eq:10}, the number of $r$-octopuses supported on $(a_1, \dots, a_r)$ in $H$ is at least
        \allowdisplaybreaks
        \begin{align*}
            &\abs{E(H^-) \cap (U_1 \times \cdots \times U_{r-1})} \times \prod_{i\in [r-1]} \left( \frac{\ve}{2^{r^2}K^{i+1}} \prod_{j\in [r]\setminus \{i\}} \abs{A_j}\right)\\ 
            &\geq \left(\frac{1}{2^{r+1}K}\prod_{i\in [r-1]}\frac{\abs{A_i}}{2^{i+2}K}\right) \times \frac{\ve^{r-1}}{2^{r^3}K^{(r^2 + r - 2)/2}} \prod_{i\in [r-1]}\prod_{j\in [r]\setminus\{i\}} \abs{A_j}\\
            &\geq \frac{\ve^{r-1}}{2^{r^3 + 2r^2} K^{(r^2 + 3r - 2)/2}} \times \left(\prod_{i\in [r]} \abs{A_i}\right)^{r-1}\\
            &\geq \frac{1}{8^{r^3}(r-1)^{(r-1)} K^{(r^2 + 5r - 4)/2}} \times \left(\prod_{i\in [r]} \abs{A_i}\right)^{r-1}.
        \end{align*}
        This completes the proof.
    \end{claimproof}
    By \Cref{clm:2}, the sets $A'_1, \dots, A'_r$ satisfy $(L2)$ of \Cref{lem:octopuses}. This completes the proof.
    
\end{proof}


\subsection{Very dense case}\label{subsec:very-dense}

\begin{lemma}\label{lem:octopuses-dense}
    Let $r \geq 2$ be an integer and $0< \frac{1}{n}, \delta \ll \ve < \frac{1}{10r}$ be real numbers.
    Let $A_1, \dots, A_r$ be disjoint vertex sets of size $n$ and $H\subseteq A_1 \times \cdots \times A_r$ be an $r$-partite $r$-uniform hypergraph with $\abs{E(H)} \geq (1 - \delta)n^r$. Then for each $i\in [r]$, there exists $A'_i\subseteq A_i$ such that
    \begin{enumerate}
        \item[$(D1)$] $\abs{A'_i} \geq \lceil(1 - \ve)\abs{A_i}\rceil$,
        \item[$(D2)$] for each $(a_1, \dots, a_r) \in A'_1 \times \cdots \times A'_r$, the number of $r$-octopuses supported on $(a_1, \dots, a_r)$ in $H$ is at least $\frac{1}{2}n^{r(r-1)}$.
    \end{enumerate}
\end{lemma}

\begin{proof}[Proof of \Cref{lem:octopuses-dense}]
    For each $i\in [r]$, let $A'_i$ be the set of vertices in $A_i$ of degree at least $(1 - \delta / \ve)n^{r-1}$.
    We claim that $A'_1, \dots, A'_r$ are the desired sets.
    By Markov's inequality, we have 
    $$
        \abs{A'_i} \geq (1 - \ve)n.
    $$ By shrinking $A'_i$, we may assume $\abs{A'_i} = \lceil (1 - \ve) n \rceil$. This implies $A'_1, \dots, A'_r$ satisfy $(D1)$ of \Cref{lem:octopuses-dense}. Let choose $\delta \defeq \frac{\ve}{10 r}$.
 
    \begin{claim}\label{clm:3}
        For each $(a_1, \dots, a_r) \in A'_1 \times \cdots \times A'_r$, the number of $r$-octopuses in $H$ that are supported on $(a_1, \dots, a_r)$ is at least $\frac{1}{2}n^{r(r-1)}$.
    \end{claim}

    \begin{claimproof}[Proof of \Cref{clm:3}]
        Let $H^- \subseteq A'_1 \times \cdots \times A'_{r-1}$ be an $(r-1)$-partite $(r-1)$ uniform hypergraph such that 
        $$
        V(H^-) = \bigcup_{i\in [r-1]} A'_i \text{ and } E(H^-) = \{e\in A'_1 \times \cdots \times A'_{r-1}: (e, a_r) \in E(H)\}.
        $$

        By our choice of $A'_r$, we have 
        $$
        \abs{E(H^-)} \geq (1 - \delta/\ve) n^{r-1} - (r-1)\ve n^{r-1} \geq (1 - \delta/\ve - r \ve)n^{r-1}.
        $$
        
        By pigeonhole principal, for each $i\in [r-1]$ and every pair $(v, w)\in A_i$, there are at least $(1 - 2\delta / \ve)n^{r-1}$ $i$-th $r$-legs on $(v, w)$.

        Hence the number of $r$-octopuses supported on $(a_1, \dots, a_r)$ is at least 
        $$
        \abs{E(H^-)} \prod_{i\in [r-1]} (1 - 2\delta / \ve)n^{r-1} \geq (1 - \delta / \ve - r\ve)(1 - 2\delta / \ve)^{r-1} n^{r(r-1)} \geq \frac{1}{2}n^{r(r-1)}.
        $$
        The last inequality holds due to our choice of $\delta$ and the assumption that $\ve < \frac{1}{10 r}$. This proves the claim.
    \end{claimproof}
    
    By \Cref{clm:3}, the sets $A'_1, \dots, A'_r$ satisfy $(D2)$ of \Cref{lem:octopuses-dense}. This completes the proof.
\end{proof}


\section{Proof of \Cref{thm:main}}\label{sec:proof}

In this section, we prove \Cref{thm:main}. The following lemma shows the connection between the size of sumsets and the $r$-octopuses.

\begin{lemma}\label{lem:many-representations}
    Let $r \geq 2$ be an integer. Let $A_1, \dots, A_r$ be subsets of an abelian group $\mathbf{G}$ and $H \subseteq A_1 \times \cdots \times A_r$ be an $r$-partite $r$-uniform hypergraph with $\abs{\bigoplus_H (A_1, \dots, A_r)} \leq C\left(\prod_{i \in [r]} \abs{A_i} \right)^{1/r}$.
    Suppose there are subsets $A_i' \subseteq A_i$ such that for each $(a_1, \dots, a_r)\in A_i\times \cdots \times A'_r$, the number of $r$-octopuses in $H$ that are supported on $(a_1, \dots, a_r)$ is at least $L\left(\prod_{i\in [r]} \abs{A_i} \right)^{r-1}$.
    Then 
    $$
    \abs{A'_1 + \cdots + A'_r} \leq \frac{C^{2r-1}}{L} \left(\prod_{i\in [r]} \abs{A_i} \right)^{1/r}.
    $$
\end{lemma}

\begin{proof}[Proof of \Cref{lem:many-representations}]
    Denote $S$ by the set $A'_1 + \cdots + A'_r$. For each $s\in S$, we choose exactly one tuple $(a_1, \dots, a_r)\in A'_1 \times \cdots \times A'_r$ such that $s = \sum_{i\in [r]} a_i$.
    Let $\calT_s$ be the collection of $r$-octopuses in $H$ that are supported on $(a_1, \dots, a_r)$. Then we have 
    \begin{equation}\label{eq:many-octopus}
        \abs{\calT_s} \geq L \left(\prod_{i\in [r]} \abs{A_i} \right)^{r-1}.
    \end{equation}

    Let $T\in \calT_s$ be an $r$-octopus supported on $(a_1, \dots, a_r)$ with
    $$
    V(T) = \{a_1, \dots, a_r\} \cup \{a^{(i)}_j: i\in [r-1],\text{ }j\in [r],\text{ }j\neq i\} \cup \{b_1, \dots, b_{r-1}\}
    $$ and
    \begin{align*}
            E(T) = &\{(a^{(i)}_1, \dots, a^{(i)}_{i+1}, a_i, a^{(i)}_{i+2}, \dots, a^{(i)}_r): i\in [r-1]\} \cup \\
            &\{(a^{(i)}_1, \dots, a^{(i)}_{i+1}, b_i, a^{(i)}_{i+2}, \dots, a^{(i)}_r): i\in [r-1]\} \cup \\
            &\{(b_1, \dots, b_{r-1}, a_r)\},
    \end{align*}
    where $a^{(i)}_j \in A_j$ for each $i\in [r-1]$ and $j\in [r]\setminus \{i\}$. 
    For each $i\in [r-1]$, let 
    \begin{align*}
        x_i^T &\defeq a_i + \sum_{j\in [r-1], j\neq i} a^{(i)}_j,\\
        y_i^T &\defeq b_i + \sum_{j\in [r-1], j\neq i} a^{(i)}_j, \text{ and}
    \end{align*}
    \begin{align*}
        z^T \defeq a_r + \sum_{i\in [r-1]} b_i.
    \end{align*}
    Observe that
    \begin{equation}\label{eq:rep}
        s = \sum_{i\in [r]} a_i = \left(\sum_{i\in [r-1]} x_i^T\right) - \left(\sum_{i\in [r-1]} y_i^T\right) + z^T.
    \end{equation}

    \begin{claim}\label{clm:4}
        The number of distinct $(c_1, \dots, c_{2r-1})\in \left(\bigoplus_H (A_1, \dots, A_r) \right)^{2r-1}$ such that
        \begin{equation}\label{eq:clm-rep}
            s = \left(\sum_{i\in [r-1]} c_i\right) - \left(\sum_{r \leq i \leq 2r-2} c_i \right) + c_{2r-1}
        \end{equation}
        is at least $L\left(\prod_{i\in [r]} \abs{A_i} \right)^{(2r-2)/r}$.
    \end{claim}

    \begin{claimproof}[Proof of \Cref{clm:4}]
        By \eqref{eq:rep}, for each $T\in \calT_s$ generates $(x_1^T, \dots, x_{r-1}^T, y_1^T, \dots, y_{r-1}^T, z^T)$ that satisfies \eqref{eq:clm-rep}.
        Let $(x_1, \dots, x_{r-1}, y_1, \dots, y_{r-1}, z) \in \left(\bigoplus_H (A_1, \dots, A_r) \right)^{2r-1}$ be a tuple such that $s = \left(\sum_{i\in [r-1]} x_i\right) - \left(\sum_{i\in [r-1]} y_i\right) + z$. Let $\calC$ be the set of $T\in \calT_s$ such that $x_i^T = x_i$, $y_i^T = y_i$ for each $i\in [r-1]$ and $z^T = z$.

        Without loss of generality, let $A_r$ have the biggest size among the sets $A_1, \dots, A_r$. Assume the group elements $a_j^{(i)}$ are already determined, where $i\in [r-1]$, $j\in [r-1]\setminus\{i\}$. Then, from the equation $x_i = a_i + \sum_{j\in [r-1], j\neq i} a_j^{(i)}$, other group elements $a^{(i)}_r$ are uniquely determined, where $i\in [r-1]$. Now, from the $x_i = a_i + \sum_{j\in [r-1], j\neq i} a_j^{(i)}$, the group elements $b^{(i)}_r$ are uniquely determined, where $i\in [r-1]$. Hence we have 
        \begin{equation}\label{eq:bounded}
            \abs{\calC} \leq \prod_{i\in [r-1]}\prod_{j\in [r-1]\setminus\{i\}} \abs{A_j} \leq \left(\prod_{i\in [r-1]} \abs{A_i} \right)^{r-2} \leq \left(\prod_{i\in [r]} \abs{A_i}\right)^{(r-1)(r-2)/r}.
        \end{equation}
        By \eqref{eq:many-octopus} and \eqref{eq:bounded}, there are at least $L\left(\prod_{i\in [r]} \abs{A_i} \right)^{(2r-2)/r}$ tuples $(c_1, \dots, c_{2r-1})$ that satisfy \eqref{eq:clm-rep}.
        This completes the proof.
    \end{claimproof}

     By \Cref{clm:4}, we have the following inequality holds.
     \begin{equation}\label{eq:final}
         \abs{A'_1 + \cdots + A'_r} L \left(\prod_{i\in [r]} \abs{A_i} \right)^{(2r-2)/r} \leq \abs{\bigoplus_H (A_1, \dots, A_r)}^{2r - 1} \leq C^{2r-1}\left(\prod_{i\in [r]} \abs{A_i} \right)^{(2r-1)/r}.
     \end{equation}
     From \eqref{eq:final}, we have 
     $$
     \abs{A'_1 + \cdots + A'_r} \leq \frac{C^{2r-1}}{L} \left(\prod_{i\in [r]} \abs{A_i} \right)^{1/r}.
     $$
     This completes the proof.
\end{proof}

We now prove \Cref{thm:main}.

\begin{proof}[Proof of \Cref{thm:main}]
    By \Cref{lem:octopuses}, we have $A'_i \subseteq A_i$ such that the following holds for each $(a_1, \dots, a_r) \in A'_i \times \cdots \times A'_r$.
    \begin{enumerate}
        \item[$(i)$] $\abs{A'_i} \geq \frac{\abs{A_i}}{2^{i+2}K}$,
        \item[$(ii)$] the number of $r$-octopuses supported on $(a_1, \dots, a_r)$ in $H$ is at least
        $$
        \frac{1}{8^{r^3}(r-1)^{(r-1)}K^{(r^2 + 5r - 4)/2}}\times \left(\prod_{i\in [r]} \abs{A_i} \right)^{r-1}.
        $$
    \end{enumerate}
    We now apply \Cref{lem:many-representations}, where $(8^{r^3}(r-1)^{(r-1)}K^{(r^2 + 5r - 4)/2})^{-1}$ plays a role of $L$. Then we have
    $$
    \abs{A'_1 + \cdots + A'_r} \leq 8^{r^3}(r-1)^{(r-1)}K^{(r^2 + 5r - 4)/2}C^{2r-1} \left(\prod_{i\in [r]} \abs{A_i} \right)^{1/r}.
    $$
    This completes the proof.
\end{proof}


\section{Proof of \Cref{thm:almost-hypergraph-BSG}}\label{sec:proof-almost}

As a direct consequence of \Cref{lem:octopuses-dense,lem:many-representations}, we obtain the following lemma, which is a weaker version of \Cref{thm:almost-hypergraph-BSG}.

\begin{lemma}\label{lem:weaker}
    Let $r\geq 2$ be an integer and $0  < \frac{1}{n}, \delta \ll \frac{1}{C}, \frac{1}{r}, \ve  < 1$ be positive real numbers.
    Let $A_1, \dots, A_r$ be subsets of an abelian group $\mathbf{G}$ with $\lvert A_1 \rvert = \cdots = \lvert A_r \rvert = n$. Suppose $H \subseteq A_1 \times \cdots \times A_r$ is an $r$-uniform hypergraph with $\lvert E(H) \rvert \geq (1 - \delta)n^r$ and $\lvert \bigoplus_H (A_1, \dots, A_r) \rvert \leq C n$.
    Then for each $i\in [r]$, there exists $A'_i \in A_i$ such that 
    \begin{enumerate}
        \item[$\bullet$] $\abs{A'_i} = \lceil (1 - \ve)n \rceil$ for each $i\in [r]$,
        \item[$\bullet$] $\abs{A_1' + \cdots + A_r'} \leq 2 C^{2r-1} n$.
    \end{enumerate}
\end{lemma}

The rest of the proof is almost identical to Shao's proof of the graph case~\cite{Shao}. The following lemma can be obtained from the arithmetic removal lemma due to Green~\cite{removal} and \cite[Proposition 1.2]{dense-model} (See~\cite[Corollary 2.5]{Shao}). 

\begin{lemma}\label{lem:removal}
    Let $r \geq 2$ be an integer and $0 < \delta \ll \frac{1}{K}, \frac{1}{r}, \ve < 1$ be positive real numbers. Let $\mathbf{G}$ be an abelian group, $X \subseteq \mathbf{G}$ be a finite subset with doubling constant at most $K$ for some $K \geq 1$, and let $A_1, \dots, A_{r+1} \subseteq X$. 
    If $\abs{(A_1 + \cdots + A_r) \cap A_{r+1}} \leq \delta \abs{X}^{r+1}$, then for each $i\in [r+1]$, there exist subsets $A'_i \subseteq A_i$ such that the following holds.
    \begin{enumerate}
        \item[$\bullet$] $\abs{A'_i} \geq \abs{A_i} - \ve \abs{X}$,
        \item[$\bullet$] $(A'_1 + \cdots + A'_r) \cap A'_{r+1} = \emptyset$. 
    \end{enumerate}
\end{lemma}

From \Cref{lem:weaker,lem:removal}, one can deduce the following lemma.

\begin{lemma}\label{lem:bsg-removal}
    Let $r, C, n, \varepsilon, \delta$ be the constants as in \Cref{thm:almost-hypergraph-BSG}, and let $H$ be the hypergraph and $A_1, \dots, A_r \subseteq \mathbf{G}$ the sets as in \Cref{thm:almost-hypergraph-BSG}.
    Let $B\subseteq \mathbf{G}$ be the subset such that $\abs{(A_1 + \cdots + A_r) \cap B} \leq \delta n^r$. Then for each $i\in [r]$, there exist subsets $A'_i \subseteq A_i$ and $B' \subset B$ such that the following holds.
    \begin{enumerate}
        \item[$\bullet$] $\abs{A'_i} \geq (1 - \ve)n$,
        \item[$\bullet$] $\abs{B'} \geq \abs{B} - \ve n$,
        \item[$\bullet$] $(A'_1 + \cdots + A'_r) \cap B' = \emptyset$. 
    \end{enumerate}
\end{lemma}

\begin{proof}[Proof of \Cref{lem:bsg-removal}]
    By \Cref{lem:weaker}, there exist $\widetilde{A}_1, \dots, \widetilde{A}_r$ with $\abs{\widetilde{A}_1} = \cdots \abs{\widetilde{A}_r} \geq (1 - \ve/3) n$ and $\abs{\widetilde{A}_1 + \cdots + \widetilde{A}_r} \leq 2C^{2r-1}n$.
    Let $\widetilde{B} = B \cap (\widetilde{A}_1 + \cdots + \widetilde{A}_r)$ and let $X = \bigcup_{i\in [r+1]} \widetilde{A}_i$. Then by Ruzsa triangle inequality, there exists a constant $C'$ which depend only on $r$ and $C$ such that 
    $$
    \abs{X + X} \leq C' \abs{X}.
    $$
    Since $\abs{X} \geq \abs{\widetilde{A}_1} \geq (1 - \ve/3)n$, the following inequality holds.
    $$
    \abs{(\widetilde{A}_1 + \cdots + \widetilde{A}_r) \cap \widetilde{B}} \leq \delta n^r \leq 2\delta \abs{X}^{r+1}.
    $$
    Then by \Cref{lem:removal}, for each $i\in [r]$, there exists $A'_i \in \widetilde{A}_i$ and $B_0 \subseteq \widetilde{B}$ such that 
    \begin{equation}\label{eq:size}
      \abs{A'_i} \geq \abs{\widetilde{A}_i} - \frac{\ve}{3(r+2C^{2r-1})}\abs{X}, \text{ } \abs{B_0} \geq \abs{\widetilde{B}} - \frac{\ve}{3(r+2C^{2r-1})}\abs{X}  
    \end{equation}
    and
    \begin{equation}\label{eq:no-solution}
        (A'_1 + \cdots + A'_r) \cap B_0 = \emptyset.
    \end{equation}
    Let $B' = (B \setminus \widetilde{B})\cup B_0$.
    Since $\abs{X} \leq \sum_{i\in [r]}\abs{\widetilde{A}_i} + \abs{\widetilde{A}_1 + \cdots + \widetilde{A}_r} \leq (r + 2C^{2r-1})n$, for each $i\in [r]$, by \eqref{eq:size}, we have 
    $$
    \abs{A'_i} \geq \abs{\widetilde{A}_i} - \frac{\ve}{3}n \geq (1 - \ve)n \text{ and } \abs{B'} \geq \abs{B} - \frac{\ve}{3}n \geq \abs{B} - \ve n.
    $$
    Lastly, by our choice of $\widetilde{B}$ and \eqref{eq:no-solution}, we have 
    $$
    (A'_1 + \cdots + A'_r) \cap B' = \emptyset.
    $$
    This completes the proof.
\end{proof}

We now prove \Cref{thm:almost-hypergraph-BSG}.

\begin{proof}[Proof of \Cref{thm:almost-hypergraph-BSG}]
    Let $B = (A_1 + \cdots + A_r) \setminus \bigoplus_H (A_1, \dots, A_r)$. By the definition of $B$, if $(a_1, \dots, a_r)\in A_1 \times \cdots \times A_r$ satisfies $\sum_{i\in [r]} a_i \in B$, then $(a_1, \cdots, a_r) \notin E(H)$. This means
    $$
    (A_1 + \cdots + A_r) \cap B \leq \delta n^r.
    $$
    By \Cref{lem:bsg-removal}, for each $i\in [r]$, there exist $A'_i\in [r]$ and $B' \in B$ with $\abs{A'_i} \geq (1 - \ve)n$ and $\abs{B'} \geq \abs{B} - \ve n$ such that $(A'_1 + \cdots + A'_r) \cap B' = \emptyset$.
    By our choice of $B$, we have 
    $$
    (A'_1 + \cdots + A'_r) \subseteq \bigoplus_{H} (A_1, \cdots, A_r) \cup (B\setminus B').
    $$
    Thus, the following inequality holds.
    $$
    \abs{A'_1 + \cdots + A'_r} \leq \abs{\bigoplus_H (A_1, \dots, A_r)} + \abs{B\setminus B'} \leq (C + \ve)n.
    $$
    This completes the proof.
\end{proof}


\section{Concluding remarks}\label{sec:concluding}

In this paper, we obtain a considerable improvement on the bounds for the hypergraph variant of Balog--Szemer\'{e}di--Gowers theorem. We note that the proof of \Cref{thm:graph-BSG} that we described in \Cref{subsec:overview} allows the non-commutative version as we can write $ab = (ab') (a'b')^{-1} (a'b)$. However, our proof does not work in a non-commutative group, because our hypergraph $r$-octopus behaves in a different way. Thus, we wonder whether there exists a non-commutative analogue of the hypergraph variant of Balog--Szemer\'{e}di--Gowers theorem as follows.

\begin{question}
    Let $r > 2$ be an integer, $\mathbf{S}$ be a non-commutative group, and $Z_1, \dots, Z_r \subseteq \mathbf{S}$ with $\abs{Z_1} = \cdots= \abs{Z_r} = n$. Let $H\subseteq Z_1 \times \cdots \times Z_r$ be a hypergraph and let $\bigotimes_H (Z_1, \dots, Z_r) \defeq \{\prod_{i\in [r]}z_i: (z_1, \dots, z_r)\in E(H)\}$. Suppose $\abs{E(H)} \geq \frac{n^r}{K}$ and $\abs{\bigotimes_H (Z_1, \dots, Z_r)} \leq C n$.
    Is it true that for each $i\in [r]$, there exists $Z'_i \subseteq Z_i$ such that 
     \begin{enumerate}
         \item[$\bullet$] $\abs{Z'_i} \geq \frac{n}{f_r(K)}$,
         \item[$\bullet$] $\abs{Z'_1\times \cdots \times Z'_r} \leq g_r(C, K)n$, 
     \end{enumerate}
     where $f_r$ and $g_r$ are polynomials whose degrees and coefficients depend only on $r$? 
\end{question}


\subsection*{Acknowledgement}
The author thanks his advisor Jaehoon Kim for his helpful advice and encouragement.
He also thanks Akshat Mudgal~\cite{Mudgal-private} for sharing his result, \Cref{thm:mudgal}.


\printbibliography

\end{document}